\documentclass[preprint,11pt]{elsarticle}
\usepackage{amsmath,amsthm,amssymb}
\DeclareMathOperator*{\esssup}{ess\,sup}
\def\RR{\hbox{I\kern-.2em\hbox{R}}}
\usepackage{fullpage}

\numberwithin{equation}{section}

\newtheorem{Th}{Theorem}[section]

\newtheorem{guess}[Th]{Theorem}
\newtheorem{uess}[Th]{Lemma}
\newtheorem{corollary}[Th]{Corollary}
\newtheorem{definition}[Th]{Definition}
\newtheorem{example}[Th]{Example}
\newtheorem{remark}[Th]{Remark}
\newtheorem{prop}[Th]{Proposition}

\begin{document}

\begin{frontmatter}

\title{On stability of linear neutral  differential equations in the Hale form}


\author[label1]{Leonid Berezansky}
\author[label2]{Elena Braverman}
\address[label1]{Dept. of Math.,
Ben-Gurion University of the Negev,
Beer-Sheva 84105, Israel}
\address[label2]{Dept. of Math. and Stats., University of
Calgary,2500 University Drive N.W., Calgary, AB, Canada T2N 1N4; e-mail
maelena@ucalgary.ca, phone 1-(403)-220-3956, fax 1-(403)--282-5150 (corresponding author)}





\begin{abstract}
We present  new explicit exponential stability conditions
for the linear scalar neutral equation with two variable coefficients and delays
$$
(x(t)-a(t)x(g(t)))'=-b(t)x(h(t)),
$$
where
$|a(t)|<1$, $b(t)\geq 0$, $h(t)\leq t$, $g(t)\leq t$, 
in the case when the delays $t-h(t)$, $t-g(t)$ are bounded, as well as an asymptotic stability
condition, if the delays can be unbounded.
\end{abstract}
 

\begin{keyword}
neutral  equations in the Hale form, neutral pantograph equation, exponential stability, solution
estimates, unbounded delays, asymptotic stability 

\noindent
{\bf AMS subject classification:} 
34K40, 34K20, 34K06, 45J05

\end{keyword}

\end{frontmatter}


\section{Introduction}

Investigation of linear neutral delay differential equations has a long history. The term ``neutral equation"
was introduced by G.~Kamenskii, and the first results were obtained by Russian mathematicians
in the 60ies, the review of them can be found in \cite{Zverkin}. Since then, many papers and monographs on the theory and applications
of neutral equations appeared, see, for example,  \cite{Elsgolts, Gil, Gop, GL, H, KolmNos,  Kuang}.

In this paper we consider the equation 
\begin{equation}\label{1}
(x(t)-a(t)x(g(t)))'+ b(t)x(h(t))=0,
\end{equation}
and call it ``the neutral differential equation in the form of Hale", due to 
essential results on this class of equations obtained by J. Hale \cite{H}. 
Another class of neutral equations including several delayed terms with
a derivative was studied in \cite{AS, KolmMysh}. 
In \cite{H} and many other papers, the authors study linear and nonlinear equations in the Hale form
under the assumption that all the parameters of equations and solutions are continuous functions.

In \cite{H}, the solution of (\ref{1}) was assumed to satisfy the integral equation
\begin{equation}\label{1add}
x(t)-a(t)x(g(t))+\int_{t_0}^t b(s)x(h(s))ds=0,
\end{equation}
which allowed to consider continuous $a,b,h,g$. We study
equation \eqref{1}, where all the functions involved in the equation, as well as solutions, are Lebesgue measurable functions, and \eqref{1add} holds. 
Such equations were investigated in the recent monograph \cite{Gil}, where in particular existence and uniqueness
results were established. We will use these results without further discussion.
The aim of the present paper is to obtain explicit asymptotic stability tests for equation \eqref{1}.

The main method to study stability for neutral equations is the construction
of Lyapunov-Krasovskii functions and functionals, see  \cite{Gop, KolmMysh, KolmNos, Shaikhet}.
Propositions~\ref{proposition3} and \ref{proposition1} below are obtained by this method.

The results of \cite[Theorem 5.1.1]{Gop} can be applied to an autonomous
neutral equation
\begin{equation}\label{04}
\left( x(t)  - a x(t-\sigma) \right)'
 =-b_0x(t)- b x(t-\tau)
\end{equation}
where
$b_0>0$, $\tau\geq 0$, $\sigma\geq 0$, $b\tau\neq 0$, $a\sigma\neq 0$.

\begin{prop}
\label{proposition3} \cite[Theorem 5.1.1]{Gop}
Assume that $b_0>0, b+b_0>0, |b| \tau < 1-|a|$. 
Then all solutions of (\ref{04}) satisfy $\displaystyle \lim_{t\rightarrow\infty}x(t)=0$.
\end{prop}

Consider (\ref{04}) with variable coefficients
\begin{equation}\label{02}
(x(t)+a(t)x(t-\sigma))'+b_0(t)x(t)+b(t)x(t-\tau)=0,
\end{equation}
where $\sigma\leq \tau$, $a,b_0,b \in C([t_0,\infty),[0,\infty))$. 

\begin{prop}\label{proposition1} \cite{AgGr}
Assume that there exist constants $p_1,p_2,q_1,q_2,a_0,A$ such that
$$
0\leq p_1\leq b_0(t)\leq p_2, ~~0\leq q_1\leq b(t)\leq q_2, ~~0\leq a(t)\leq a_0<1,~~ |a'(t)|\leq A,
$$
$\sigma\leq \tau$, $a,b_0,b \in C([t_0,\infty),[0,\infty))$, 
and $c$ is differentiable with a locally bounded derivative.

If at least one of the following conditions 

a) $p_1+q_1>(p_2+q_2)(a_0+q_2\tau)$;

b) $p_1>q_2+ a_0(p_2+q_2)$
\\
holds then every solution of (\ref{02}) satisfies $\displaystyle \lim_{t \to +\infty} x(t) = 0$.
\end{prop}

The next two stability results are based on a deep analysis of neutral equation
\eqref{1} with constant delays 
\begin{equation}\label{03}
\left( x(t)-a(t)x(t-\sigma) \right)'+b(t)x(t-\tau)=0.
\end{equation}

\begin{prop}\label{proposition2} \cite{Yu}
Let $\tau, \sigma>0$, $a,b \in C([t_0,\infty), \RR)$, $b(t)\geq 0$.
If 
$$
\int_{t_0}^{\infty} b(s)ds=+\infty,~ |a(t)|\leq a_0< 1,~~ 
\limsup_{t\rightarrow\infty} \int_{t-\tau} ^t b(s)ds<\frac{3}{2}-2a_0(2-a_0)
$$
then equation (\ref{03}) is asymptotically stable.
\end{prop}

Proposition \ref{proposition2} is a nice result, since in the non-neutral case $a(t)\equiv 0$ 
it leads to a sharp stability test with the famous constant $\frac{3}{2}$.

There are several improvements and extensions of Proposition \ref{proposition2},
in particular, the following result from \cite{TangZou}.

\begin{prop}\label{proposition2a} \cite{TangZou}
Let $\displaystyle \int_{t_0}^{\infty} b(s)ds=+\infty$ and $|a(t)|\leq a_0 < 1$.
Assume that at least one of the following conditions holds:
\vspace{2mm}

a) $\displaystyle a_0<\frac{1}{4},~~\limsup_{t\rightarrow\infty} \int_{t-\tau}^t b(s)ds<\frac{3}{2}-2a_0$; 
\vspace{2mm}

b) $\displaystyle \frac{1}{4}\leq a_0 <\frac{1}{2},~~\limsup_{t\rightarrow\infty} \int_{t-\tau} ^t b(s)ds<\sqrt{2(1-2 a_0)}$.
\vspace{2mm}

Then equation \eqref{03} is asymptotically stable.
\end{prop}

Every method used to investigate stability has its advantages and limitations. Some stability tests were 
obtained by an advanced analysis of specific equations, such as Propositions~\ref{proposition2} and \ref{proposition2a}.
Such results usually have conditions close to the best possible ones, but, generally, this method 
fails for equations with time-dependent delays.

The method of Lyapunov-Krasovskii functions and functionals works for most known
classes of functional differential equations. Unfortunately, it is not easy to apply this method for equations with variable, in particular with unbounded, delays.

Equations with proportional delays $g(t)=\mu t$, $h(t)=\lambda t$ and, more generally, unbounded delays are usually called pantograph
or generalized pantograph equations. One of the first and nice results for this class of equations
was obtained in \cite{Kato}.

\begin{prop}\label{proposition2b} \cite{Kato}
Equation 
$$
\dot{x}(t)=ax(t)+bx(\lambda t), ~~0<\lambda<1,
$$
is asymptotically stable if and only if $a<0$, $|b|<|a|$.
\end{prop}

Various other results on asymptotic stability and asymptotic behavior of solutions
for neutral pantograph equations were obtained in \cite{Derfel, Iserles, Iserles1, Iserles2, Iserles3, Liu}.
A good review on stability theory for pantograph neutral equations can be found in the monograph \cite{Bellen}.

In the present paper, we consider scalar linear non-autonomous pantograph neutral equations. 

Using the Bohl-Perron theorem, stability tests for all classes of linear functional differential equations can be obtained. 
The advantage of this method is that, instead of studying stability, it is sufficient to estimate either the norm or the spectral radius
of a linear operator in some functional spaces on the half-line. Explicit stability results were established by this method
in \cite{AS,BB2_2018} and in the monograph \cite{AS} for a linear neutral equation which is different from \eqref{1}. 
To the best of our knowledge, 
this method  
is applied to equation \eqref{1} for the first time.
The Bohl-Perron theorem for this class of equations can be found in \cite{Gil}.

The paper is organized as follows. 
Section 2 presents  definitions, assumptions and auxiliary statements.
In Section~3, the main stability results for equation (\ref{1}) 
are justified.
Section~4 contains examples and discussion.


\section{Auxiliary Results}

We consider (\ref{1}) under the following assumptions:
\\
(a1) $a, b, g, h$ are Lebesgue measurable  essentially
bounded functions on $[0,+\infty)$;\\
(a2) $ \mbox{ess}\sup_{t\geq t_0}  |a(t)|\leq a_0<1$  for some $t_0\geq 0$, $b(t)\geq 0$;\\ 
(a3) $g(t)\leq t$, $\displaystyle \lim_{t\rightarrow +\infty}g(t)= +\infty$, $mes~ U=0\Longrightarrow mes~ g^{-1}(U)=0$,
where $mes~U$ is  the Lebesgue measure of the set $U$;
\\
(a4) $h(t)\leq t$, $\displaystyle \lim_{t\to +\infty}h(t)=+\infty$, $mes~ U=0 \Longrightarrow mes~ h^{-1}(U)=0$.

Together with  (\ref{1}) we consider for each $t_0 \geq 0$ an initial value problem
\begin{equation}
\label{3}
(x(t)-a(t)x(g(t)))'+b(t)x(h(t))=f(t), ~~t\geq t_0,~~
x(t)=\varphi(t), ~ t \leq t_0,
\end{equation}
where 
\\
(a5) $f:[t_0,+\infty)\rightarrow {\mathbb R}$ is Lebesgue measurable locally essentially
bounded, $\varphi :(-\infty,t_0)\rightarrow {\mathbb R}$ is a Borel measurable and bounded function.

In some of our main results, we assume that the delays are bounded:
\\
(a6) $t-g(t)\leq \delta$, $t-h(t) \leq \tau$ for $t \geq t_0$ and some $\delta>0$, $\tau>0$ and $t_0 \geq 0$.

\begin{definition} 
A Lebesgue measurable function $x: {\mathbb R} \rightarrow {\mathbb R}$  is called {\bf a solution of problem}  (\ref{3}) if 
it is locally essentially bounded on $[0,+\infty)$,
$x(t)-a(t)x(g(t))$ is locally absolutely continuous, 
$x$ satisfies the  equation in \eqref{3} (i.e. \eqref{1add} with the right-hand side $\int_{t_0}^t f(s)ds$)  
for almost all $t\in [t_0,+\infty)$, and the initial condition
in (\ref{3}) holds for $t\leq t_0$.
\end{definition}

There exists a unique solution of problem (\ref{3}), see \cite{Gil} for conditions (a1)-(a4) and \cite{H}
for continuous $a,b,g,h$.

Consider the initial value problem for the equation with one non-neutral delay term
\begin{equation}
\label{5}
x^{\prime}(t)+b(t)x(h(t))=f(t),~t \geq t_0,~~x(t)=0,~t\leq t_0,
\end{equation}
where $b(t), f(t)$ and  $h(t) \leq t$ are Lebesgue measurable locally bounded functions.

\begin{definition} 
For each $s\geq t_0$ the solution $X(t,s)$ of the problem
\begin{equation}
\label{6}
x^{\prime}(t)+b(t)x(h(t))=0,~t \geq t_0,~ ~x(t)=0,~t<s,~x(s)=1
\end{equation}
is called {\bf a fundamental function of equation}  (\ref{5}).
We assume $X(t,s)=0$ for  $0\leq t<s$.
\end{definition} 

\begin{uess}  \cite{AS}
\label{lemma2}
The solution of
problem  (\ref{5})  can be presented as
$\displaystyle x(t)=\int_{t_0}^t X(t,s)f(s)ds$.
\end{uess}

\begin{definition} 
Equation (\ref{1}) is {\bf (uniformly) exponentially stable} 
if there are $M>0$, $\gamma>0$ such that 
the solution of problem  (\ref{3}) with $f \equiv 0$ has the estimate 
$\displaystyle |x(t)|\leq M e^{-\gamma (t-t_0)} \!\!\! \sup_{t \in (-\infty,  t_0]} |\varphi(t)|$ for $t\geq t_0$,
 where $M$ and $\gamma$ do not depend on $t_0 \geq 0$ and $\varphi$.
\end{definition}

All our main results are based on the Bohl-Perron theorem which is stated below.

\begin{uess}\label{lemma3}\cite[Theorem 6.1]{Gil}
Assume that (a1)-(a4),(a6) hold, and  the solution of the problem 
\begin{equation}\label{10}
(x(t)-a(t)x(g(t)))'+b(t)x(h(t))=f(t),~t \geq t_0,~ ~x(t)=0,~t\leq t_0
\end{equation}
is bounded on $[t_0,+\infty)$ for any 
essentially bounded function $f$ on $[t_0,+\infty)$. 
Then equation (\ref{1}) is exponentially stable.
\end{uess}

\begin{remark}
\label{remark1a}
In Lemma~\ref{lemma3} we can consider
boundedness of solutions not for all essentially bounded functions $f$ on $[t_0,+\infty)$  but only
for essentially bounded functions $f$ on $[t_1,+\infty)$ that vanish on $[t_0,t_1)$ for any fixed $t_1>t_0$, see \cite{BB3}.
We will further apply this fact in the paper without an additional reference.
\end{remark}

Consider now a linear equation with a single delay and a non-negative coefficient
\begin{equation}\label{9}
x^{\prime}(t)+b(t)x(h_0(t))=0, ~~b(t)\geq 0,~~ 0\leq t-h_0(t)\leq \tau_0,
\end{equation}
and let $X_0(t,s)$ be its fundamental function.

\begin{uess}\label{lemma4}\cite{BB3}
Assume that $X_0(t,s)>0$ , $t\geq s\geq t_0$. Then 
$\displaystyle
\int_{t_0+\tau_0}^t X_0(t,s) b(s)ds\leq 1.
$
\end{uess}

\begin{uess}\label{lemma5}\cite{BB3,GL}
Assume that there is $t_0\geq 0$ such that
$\displaystyle
 \int_{h_0(t)}^t b(s) ds\leq \frac{1}{e}$ for any $t\geq t_0
$. Then $X_0(t,s)>0$ for $t\geq s\geq t_0$.
If in addition $b(t)\geq b_0>0$ then equation (\ref{9}) is 
exponentially stable.
\end{uess}

For a fixed bounded interval $I=[t_0,t_1]$, consider the space $L_{\infty}[t_0,t_1]$ of all essentially bounded on $I$
functions with the 
norm $|y|_I= \esssup_{t\in I} |y(t)|$.
Denote for an unbounded interval $$\|f\|_{[t_0,+\infty)}=\esssup_{t\geq t_0} |f(t)|,$$ by $E$ the identity operator.
Define the operator $S$ 
on the space $L_{\infty}[t_0,t_1]$ as 
$$\displaystyle 
(Sy)(t)=\left\{\begin{array}{ll}
a(t)y(g(t)),& g(t)\geq t_0,\\
0,& g(t)<t_0.\\
\end{array}\right. 
$$


\begin{uess}\label{lemmaS} \cite{ABR}
Let $a,g$ satisfy (a1) and (a3), respectively. 
If $\|a\|_{[t_0,+\infty)}\leq a_0<1$ then $E-S$ is invertible in the space $L_{\infty}[t_0,+\infty)$,  
and the operator norm satisfies
\begin{equation}
\label{star}
\displaystyle \|(E-S)^{-1}\|_{L_{\infty}[t_0,+\infty)\to L_{\infty}[t_0,+\infty)}\leq \frac{1}{1-\|a\|_{[t_0,+\infty)}}.
\end{equation}
\end{uess}

\section{Stability Results}


Consider initial value problem (\ref{10})
with  $\|f\|_{[t_0,+\infty)}< +\infty$. 
First, let us estimate its solution and the expression under the sign of the derivative.


\begin{uess}\label{lemma9}
Suppose (a1)-(a4) hold.
A solution of (\ref{10}) and the derivative of $y(t)=x(t)-a(t)x(g(t))$ satisfy 
on any interval $I=[t_0,t_1]$, $t_1>t_0$,
\begin{equation}
\label{star1}
|x|_I\leq \frac{1}{1-\|a\|_{[t_0,+\infty)}}|y|_I,~~
|y^{\prime}|_I\leq \frac{\|b\|_{[t_0,+\infty)}}{1-\|a\|_{[t_0,+\infty)}}|y|_I+\|f\|_{[t_0,+\infty)}.
\end{equation}
\end{uess}
\begin{proof}
We have for $t\in I$ by Lemma~\ref{lemmaS},
$$
x(t)  =  (E-S)^{-1}y(t),~ |x|_I\leq \|(E-S)^{-1}\|_{L_{\infty}[t_0,t_1]\rightarrow L_{\infty}[t_0,t_1]}|y|_I
\leq \frac{1}{1-\|a\|_{[t_0,+\infty)}}|y|_I,
$$
\begin{eqnarray*}
|y^{\prime}(t)| & \leq &  |b(t)|~|x(h(t))|+\|f\|_{[t_0,+\infty)}
\\
& \leq & \|b\|_{[t_0,+\infty)}|x|_I+\|f\|_{[t_0,+\infty)}
\leq \frac{\|b\|_{[t_0,+\infty)}}{1-\|a\|_{[t_0,+\infty)}}|y|_I+\|f\|_{[t_0,+\infty)}.
\end{eqnarray*}
\end{proof}

\begin {guess}\label{theorem1}
Assume that (a1)-(a4),(a6) hold and 
there exists $t_0\geq 0$ such that for $t\geq t_0$
\begin{equation}
\label{ice1}
0<b_0\leq b(t),~~  \int_{h(t)}^t b(s)~ds \leq \frac{1}{e}\,
\end{equation}
and
\begin{equation}\label{11}
\|a\|_{[t_0,+\infty)}<\frac{1}{2}.
\end{equation}
Then equation (\ref{1}) is exponentially stable.
\end{guess}
\begin{proof}
We will prove that a solution of (\ref{10}) for any $\|f\|_{[t_0,+\infty)}< +\infty$ 
(satisfying  in addition $f(t)=0$ for $t \in [t_0,t_0+\tau)$) is  bounded on $[t_0,+\infty)$. 
Let $Y_1(t,s)$ be the fundamental function of the equation
\begin{equation}\label{13}
y^{\prime}(t)+b(t)y(h(t))=0.
\end{equation}

By \eqref{ice1} and Lemma~\ref{lemma5}, $Y_1(t,s)>0$ for any  
$t\geq s\geq t_0$.
Also, $b(t)\geq b_0>0$ 
implies exponential stability of equation (\ref{13}), and 
$Y_1(t,s)$ has an exponential estimate.

Let $y(t)=x(t)-a(t)x(g(t))$, then
$b(t)x(h(t))=b(t)y(h(t))+b(t)a(h(t))x(g(h(t)))$, and (\ref{10}) can be rewritten in the form
$$
y^{\prime}(t)+b(t)y(h(t))= -b(t)a(h(t))x(g(h(t)))+f(t), ~y(t)=0, ~t\leq t_0.
$$
By Lemma \ref{lemma2}, 
$$
y(t)=-\int_{t_0}^t Y_1(t,s) b(s)a(h(s))x(g(h(s)))ds+f_1(t),
$$
where
$
f_1(t)=\int_{t_0}^t Y_1(t,s) f(s)ds.
$
Since  $Y_1(t,s)$ has an exponential estimate and $f$ is bounded on $[t_0,+\infty)$, $\|f_1\|_{[t_0,+\infty)}< +\infty$.

Denote $I=[t_0,t_1]$. By Lemma \ref{lemma4}, using the fact that $x(t)=y(t)=0$ for $t\in [t_0,t_0+\tau]$ 
and the first estimate in \eqref{star1}, we get
$$
|y|_I\leq \|a\|_{[t_0,+\infty)}|x|_I+\|f_1\|_{[t_0,+\infty)}\leq \frac{\|a\|_{[t_0,\infty)}}{1-\|a\|_{[t_0,+\infty)}}|y|_I+\|f_1\|_{[t_0,+\infty)}.
$$

By (\ref{11}) we have $|y|_I\leq M$, where $M$ does not depend on the interval $I$.
Then, also by the first estimate in \eqref{star1}, $|x|_I\leq \widetilde{M}$, where $\widetilde{M}$ does not depend on the interval $I$. Hence $|x(t)|\leq \widetilde{M}$ for $t \geq t_0$.
By Lemma \ref{lemma3}, equation (\ref{1}) is exponentially stable.
\end{proof}

Let $u^+=\max\{ u,0 \}$. 

\begin{guess}
\label{theorem2}
Assume that (a1)-(a4),(a6) are satisfied,  $b(t) \geq b_0>0$ and for some $t_0 \geq 0$ at least one of the following conditions holds:
%
\begin{equation}\label{a}
\left\|\frac{b-\beta}{\beta}\right\|_{[t_0,+\infty)}
+ \left\|\frac{b}{\beta}\right\|_{[t_0,+\infty)}\frac{\|a\|_{[t_0,+\infty)}}{1-\|a\|_{[t_0,+\infty)}}
 < 1, \mbox{~~where~~~} \beta(t)= \min \left\{ b(t),\frac{1}{\tau e}\right\};
\end{equation}
%
\begin{equation}\label{b}
 \displaystyle 
\| b\|_{[t_0,+\infty)} 
 \left\| \left( t-h(t)- \frac{1}{\|b\|_{[t_0,+\infty)}e} \right)^+ \right\|_{[t_0,+\infty)} < 1- 2\| a \|_{[t_0,+\infty)}.
\end{equation}
Then equation (\ref{1}) is exponentially stable.
\end{guess}
\begin{proof}
Assume that  \eqref{a} holds. Consider problem (\ref{10}) with 
$\|f\|_{[t_0,+\infty)}< +\infty$ and $f(t)=0$ for 
$t \leq t_0+\tau$.
Denote
$\displaystyle \beta(t) := \min \left\{ b(t),\frac{1}{\tau e}  \right\}$ as in (\ref{a}) and
$\displaystyle b_1 := \min\left\{ b_0, \frac{1}{\tau e}  \right\}>0.$
Then $0< b_1 \leq \beta(t)\leq b(t)$ and
$\displaystyle \int_{h(t)}^t \beta(s)ds\leq\frac{1}{e}$.
Similarly to the proof of the previous theorem, (\ref{10}) can be rewritten as
$$
y^{\prime}(t)+\beta(t)y(h(t))=-(b(t)-\beta(t))y(h(t))-b(t)a(h(t))x(g(h(t)))+f(t),~y(t)=0,~t\leq t_0.
$$

Let $Y_2(t,s)$ be the fundamental function of the equation
\begin{equation}\label{14}
y^{\prime}(t)+\beta(t)y(h(t))=0.
\end{equation}
By Lemma \ref{lemma5},  $Y_2(t,s)>0$ and equation (\ref{14}) is exponentially stable.

Let $I=[t_0,t_1]$. We have 
$$
y(t)= \left.\left. \int_{t_0}^t Y_2(t,s) \right[-(b(s)-\beta(s))y (h(s)) -b(s)a(h(s))x(g(h(s)))\right]ds+f_2(t),
$$
where $f_2(t)=\int_{t_0}^t Y_2(t,s)f(s)ds$ and $\|f_2\|_{[t_0,+\infty)}< +\infty$.
Then 
$$\displaystyle
|y(t)|\leq \left.\left. \int\limits_{t_0}^t  Y_2(t,s)\beta(s)\right[ \left| \frac{b(s)-\beta(s)}{\beta(s)}  \right| 
|y (h(s)) |
+ \left| \frac{b(s)a(h(s))}{\beta(s)} \right| | x(g(h(s)))|
\right]ds +\|f_2\|_{[t_0,+\infty)}.$$
Hence, first by Lemma~\ref{lemma4} and then by \eqref{star1}, 
$$
|y|_I \leq  \left(\left\|\frac{b-\beta}{\beta}\right\|_{[t_0,+\infty)}\right) |y|_I +
\left(\|a\|_{[t_0,+\infty)}\left\|\frac{b}{\beta}\right\|_{[t_0,+\infty)}\right)|x|_I
+M_1
$$$$
\leq \left( \left\|\frac{b-\beta}{\beta}\right\|_{[t_0,+\infty)}
+\frac{\|a\|_{[t_0,+\infty)}}{1-\|a\|_{[t_0,+\infty)}}\left\|\frac{b}{\beta}\right\|_{[t_0,+\infty)}
\right) |y|_I +M_2
$$
for some finite $M_1>0$, $M_2>0$.
Condition (\ref{a}) 
implies $|y|_I<M$, where $M$ does not depend on the interval $I$. Hence $\|y\|_{[t_0,+\infty)}< +\infty$,
therefore by \eqref{star1},  $\|x\|_{[t_0,+\infty)}< +\infty$.
Thus by Lemma \ref{lemma3}, equation~(\ref{1}) is exponentially stable.

Next, assume that (\ref{b}) holds. 
Denote $$\displaystyle 
h_0(t)=\max\left\{ h(t), t-\frac{1}{\|b\|_{[t_0,+\infty)}e}  \right\}.$$
Then 
$$\displaystyle \int_{h_0(t)}^t b(s)ds \leq \frac{1}{e}\, ,~
h_0(t)\geq h(t), ~~
|h(t)-h_0(t)|=\left( t-h(t)- \frac{1}{\|b\|_{[t_0,+\infty)}e} \right)^+.$$


Problem  (\ref{10}) can be rewritten as
$$
y^{\prime}(t)+b(t)y(h_0(t))=b(t)\int\limits_{h(t)}^{h_0(t)} \!\!\! y^{\prime}(s)ds-b(t)a(h(t))x(g(h(t)))+f(t),~y(t)=0,~ t\leq t_0.
$$
Let $Y_3(t,s)$ be the fundamental function of the equation
\begin{equation}\label{17a}
y^{\prime}(t)+b(t)y(h_0(t))=0,
\end{equation}
where by Lemma~\ref{lemma5},   $Y_3(t,s)>0$ and equation (\ref{17a}) is exponentially stable.

For  $I=[t_0,t_1]$, we have
$$
y(t)=\int_{t_0}^t Y_3(t,s)b(s)\left(\int_{h(s)}^{h_0(s)}y^{\prime}(\xi)d\xi-a(h(s))x(g(h(s)))\right)
ds+f_3(t),
$$
where $\displaystyle f_3(t)=\int_{t_0}^t Y_3(t,s)f(s)ds$ and $\|f_3\|_{[t_0,+\infty)}< +\infty$. 
Lemma~\ref{lemma4} and \eqref{star1} imply
\begin{eqnarray*}
|y|_I &\leq & \|h_0-h\|_{[t_0,+\infty)}|y^{\prime}|_I + \| a \|_{[t_0,\infty)}|x|_I
+\|f_3\|_{[t_0,+\infty)} \\
& \leq &  \left( \left\| \left( t-h(t)- \frac{1}{\|b\|_{[t_0,\infty)}e} \right)^+ \right\|_{[t_0,+\infty)} 
\frac{\|b\|_{[t_0,+\infty)}}{1-\|a\|_{[t_0,+\infty)}} +
\frac{\|a\|_{[t_0,+\infty)}}{1-\|a\|_{[t_0,+\infty)}}\right)|y|_I+ M_3
\end{eqnarray*}
for some $M_3>0$. 
Inequality (\ref{b}) yields that
$\|y\|_{[t_0,+\infty)}\leq M$, where $M$ does not depend on the interval $I$, thus $\|x\|_{[t_0,+\infty)}< +\infty$,
and therefore equation (\ref{1}) is exponentially stable.
\end{proof}

\begin{corollary}\label{corollary2a}
Assume that (a1)-(a4),(a6) are satisfied, and at least one of the following conditions holds for $t \geq t_0$: 
\vspace{2mm}

a) $\displaystyle b(t) \geq \frac{1}{\tau e}$ and
$\displaystyle \left.\left. \tau\|b\|_{[t_0,+\infty)}<  \frac{2}{e}\right( 1-\|a\|_{[t_0,+\infty)} \right)$;
\vspace{2mm}

b) $b(t)\geq b_0 >0, \displaystyle t-h(t)\geq \frac{1}{\|b\|_{[t_0,+\infty)}e}$,   $\displaystyle \tau \|b\|_{[t_0,+\infty)}<
1+\frac{1}{e}-2\|a\|_{[t_0,+\infty)}$. 
\vspace{2mm}

Then equation (\ref{1}) is exponentially stable.
\end{corollary}

\begin{proof}
Conditions in
a)  of the corollary yield that, in the proof of Theorem~\ref{theorem2},
$$\beta(t)=\frac{1}{\tau e}\, ,~~
\left\|b-\beta \right\|_{[t_0,+\infty)}=
\|b\|_{[t_0,+\infty)}-\frac{1}{\tau e}.
$$
Hence, after some simple calculations,  condition a) of the corollary implies (\ref{a}) of Theorem~\ref{theorem2}.

Next, assume that $t-h(t)\geq \frac{1}{\|b\|_{[t_0,+\infty)}e}$. Then
\begin{eqnarray*}
& & \left\| \left( t-h(t)- \frac{1}{\|b\|_{[t_0,+\infty)}e} \right)^+ \right\|_{[t_0,+\infty)}
 =\left\|  t-h(t)- \frac{1}{\|b\|_{[t_0,+\infty)}e}\right\|_{[t_0,+\infty)}
\\
& = & \|t-h(t)\|_{[t_0,+\infty)}-\frac{1}{\|b\|_{[t_0,+\infty)}e}  \leq  \tau-\frac{1}{\|b\|_{[t_0,+\infty)}e}.
\end{eqnarray*}

The inequality
$\displaystyle
 \|b\|_{[t_0,+\infty)}\left(  \tau-\frac{1}{\|b\|_{[t_0,+\infty)}e}\right)
<1-2\|a\|_{[t_0,+\infty)} 
$
in (\ref{b}) 
is equivalent to the last inequality in b).
\end{proof}

Considering $b(t) \equiv b$ with the cases $\displaystyle t-h(t) \geq \frac{1}{eb}$ and $\displaystyle b\geq \frac{1}{\tau e}$  only, we get the following 
result.

\begin{corollary}\label{corollary2b}
Assume that (a1)-(a4),(a6) are satisfied, $b(t)\equiv b>0$,  and for some $t_0 \geq 0$,
for $t \geq t_0$, either $\displaystyle \frac{1}{e}\leq b\tau< \left.\left. \frac{2}{e} \right( 1-\|a\|_{[t_0,+\infty)} \right)$ or $\displaystyle \frac{1}{e}\leq b(t-h(t)) \leq b\tau <1+\frac{1}{e}-2\|a\|_{[t_0,+\infty)}$.

Then equation (\ref{1}) is exponentially stable.
\end{corollary}

In the following theorem, 
the delays in equation (\ref{1}) are not assumed to be bounded. Instead of exponential stability, we deduce
integral asymptotic stability conditions.

\begin{guess}\label{theorem2a}
Let (a1)-(a4) hold, $b(t)\geq 0$, $\displaystyle \int_0^{+\infty} b(s) ds=+\infty$, $b(t)\neq 0$ almost everywhere,
\begin{equation}\label{1a}
\limsup_{t\rightarrow +\infty} \int_{g(t)}^t b(\xi)d\xi< +\infty,~
\limsup_{t\rightarrow +\infty} \int_{h(t)}^t b(\xi)d\xi< +\infty
\end{equation}
and at least one of the following conditions holds for $t \geq t_0$:

a) 
$\displaystyle
\int_{h(t)}^t b(\xi)d\xi\leq \frac{1}{e} \, ,~ \|a\|_{[t_0,+\infty)}<\frac{1}{2};
$

b)
$\displaystyle
\frac{1}{e}< \int_{h(t)}^t b(\xi)d\xi< 1+\frac{1}{e} -2\|a\|_{[t_0,+\infty)}.  
$

Then equation (\ref{1}) is asymptotically stable.
\end{guess}
\begin{proof}
Let ${\displaystyle s=p(t):=\int_{t_0}^t b(\tau)d\tau,~ z(s)=x(t)}$,
where $p(t)$ is a strictly increasing function.
Then we introduce $\tilde{a}(s), \tilde{h}(s)$ and $\tilde{g}(s)$
as follows: 
$$
\tilde{a}(s)=a(t), ~x(h(t))=z(\tilde{h}(s)), ~\tilde{h}(s)\leq s, ~ \tilde{h}(s)=\int_{t_0}^{h(t)} 
b(\tau)d\tau, ~ s-\tilde{h}(s)=\int_{h(t)} ^t b(\tau)d\tau,
$$ 
$$
\tilde{g}(s)=\int_{t_0}^{g(t)} b(\tau)d\tau,~ s-\tilde{g}(s)=\int_{g(t)} ^t b(\tau)d\tau, ~\tilde{g}(s)\leq s.
$$
Then
$$
\left.\left.\left.\left.\left.\left. \frac{d}{dt}\right( x(t)-a(t)x(g(t))\right)= \frac{d}{ds} \right( z(s)-\tilde{a}(s)z(\tilde{g}(s)) \right) \frac{ds}{dt}=b(t) \frac{d}{ds}\right( z(s)-\tilde{a}(s)z(\tilde{g}(s)) \right). 
$$
Equation (\ref{1}) can be rewritten in the form
\begin{equation}\label{2a}
(z(s)-\tilde{a}(s)z(\tilde{g}(s)))'=-z(\tilde{h}(s)).
\end{equation}
By inequalities (\ref{1a}), equation  (\ref{2a}) involves bounded delays. 
If $x(t)$ is a solution of  (\ref{1}) then $z(s)=x(t)$ is a solution of  (\ref{2a}). 

Theorem~\ref{theorem1} and condition a) of the theorem, as well as Part b) of Corollary \ref{corollary2b} and condition b) of the theorem 
imply that 
equation (\ref{2a}) is exponentially stable.
Hence (\ref{1}) is stable and $\displaystyle \lim_{s\rightarrow +\infty} z(s)=\lim_{t\rightarrow +\infty} x(t)=0$,
i.e. (\ref{1}) is asymptotically stable.
\end{proof}

As an application of Theorem~\ref{theorem2a}, consider the pantograph version of equation~\eqref{1} 
\begin{equation}
\label{eq_pant}
(x(t)-a(t)x(\mu t))'=-b(t)x(\lambda t), \quad \mu, \lambda \in (0,1).
\end{equation}

\begin{corollary}
\label{cor_pant}
Assume that (a1)-(a2) hold, $b(t)\geq 0$, $\displaystyle \int_0^{+\infty} b(s) ds= +\infty$, $b(t)\neq 0$ almost everywhere,
and at least one of the following conditions holds for $t \geq t_0$:
\\
a) 
$\displaystyle
\int_{\lambda t}^t b(\xi)d\xi\leq \frac{1}{e}$, $\displaystyle \|a\|_{[t_0,\infty)}<\frac{1}{2}$;
\\
b)
$\displaystyle \frac{1}{e}< \int_{\lambda t}^t b(\xi)d\xi< 1+\frac{1}{e} -2\|a\|_{[t_0,+\infty)}$.

Then equation (\ref{1}) is asymptotically stable.

If in addition there exist $\nu_1$,$\nu_2$, $\nu_2>\nu_1>0$ such that for $t \geq t_0 >0$,
\begin{equation}
\label{reviewer_add}
\ln(\nu_1 t) \leq 
\int_{t_0}^t b(\xi)d\xi \leq \ln(\nu_2 t) 
\end{equation}
then there are $t_1 \geq t_0$, $M_1>0$ and $\gamma>0$ such that 
\begin{equation}
\label{reviewer_add_1}
|x(t)| \leq M_1 t^{-\gamma}, \quad t \geq t_1. 
\end{equation}
\end{corollary}

\begin{proof}
The only assumption that we have to check is that under either a) or b), \eqref{1a} holds.
Both a) and b) imply $\displaystyle \int_{\lambda t}^t b(\xi)d\xi< 1+\frac{1}{e}<+\infty$ for $t \geq t_0$.
The only inequality that we have to justify is the first inequality in \eqref{1a}.
If $\mu \geq \lambda$ then it is obvious. Let $\mu<\lambda$; as $\mu, \lambda \in (0,1)$,
there is an integer $k$ such that $\lambda^k<\mu$. Instead of $t_0$, consider $t_0^{\ast}= t_0 \lambda^{-k}$. 
Then for $t \geq t_0^{\ast}$,
$$ \int_{\mu t}^t b(\xi)d\xi \leq \int_{\lambda^k t}^t b(\xi)d\xi 
= \int_{\lambda^k t}^{\lambda^{k-1} t} b(\xi)d\xi + \int_{\lambda^{k-1} t}^{\lambda^{k-2} t} b(\xi)d\xi + \dots + \int_{\lambda t}^t b(\xi)d\xi
\leq k \left( 1+\frac{1}{e} \right),$$
which immediately implies the first inequality in \eqref{1a}.

Let in addition \eqref{reviewer_add} hold. Then there exists $t_1 \geq t_0$ such that 
\begin{equation}
\label{rev_add_2}
\ln (\nu t) \leq t, \quad t \geq t_1.
\end{equation}
The assumptions of the corollary imply that $z(s)=x(t)$, with $\displaystyle s=p(t):=\int_{t_0}^t b(\tau)d\tau$, is uniformly exponentially stable, see the proof of Theorem~\ref{theorem2a}. Note that $p(t_0)=0$.
Thus there are $M>0$, $\gamma>0$ such that 
\begin{equation}
\label{rev_add_3}
|x(t)| \leq M e^{-\gamma p(t)}.
\end{equation}
Since $p(t)$ is monotone increasing and the expression in the right-hand side is decreasing in $t$, inequality \eqref{rev_add_3} holds 
for $x(r)$, $r \geq p(t)$ instead of $x(p(t))$ in the left-hand side.  
By \eqref{reviewer_add} and \eqref{rev_add_2},
$$ 
p(t) \leq \ln(\nu_2 t) \leq t, \quad t \geq t_1.
$$
Thus
$$
|x(t)| \leq M e^{-\gamma p(t)} \leq Me^{-\gamma \ln(\nu_1 t)} =M \left( \nu_1 t \right)^{-\gamma} =M_1 t^{-\gamma},
$$
where $M_1 = M \nu_1^{-\gamma}$, which concludes the proof.
\end{proof}

\begin{remark}
Note that \eqref{reviewer_add} implies boundedness of $\int_{\lambda t}^t b(\tau)d\tau$.
\end{remark}

\section{Examples and Discussion}

First, we illustrate the results of the present paper with three examples: one for an equation with constant delays
and variable coefficients, one for a pantograph equation and one for an equation, where one of the delays is growing faster than for any 
pantograph equation.

\begin{example}
\label{ex1}
Consider the equation
\begin{equation}
\label{ex1eq1}
\left( x(t)- a(t)x(t-\sigma) \right)'=-\alpha (1+0.1\cos t)x(t-\pi),
\end{equation}
where $|a(t)| \leq a_0<\frac{1}{2}$. Let $a_0=0.49$. Then in Part b) of Proposition~\ref{proposition2a} \cite{TangZou},
$$\limsup_{t\rightarrow +\infty} \int_{t-\pi} ^t b(s)ds= \alpha(\pi+0.2)<0.2,$$
or $\alpha<0.05985$, while Theorem~\ref{theorem1} implies exponential stability whenever
$$
\int_{t-\pi}^t b(s)~ds \leq \alpha(\pi+0.2) \leq \frac{1}{e}\, ,
$$
or $\alpha \leq 0.11$. For $\alpha \in (0.05985,0.11]$, Theorem~\ref{theorem1} establishes exponential stability,
while Proposition~\ref{proposition2a} fails.

Next, let $a_0=0.46$, then the condition in Proposition~\ref{proposition2a} becomes $\alpha(\pi+0.2)<0.4$,
or $\displaystyle \alpha < \alpha_0 \approx 0.1197$. Part b) of Corollary~\ref{corollary2a} implies exponential stability
for $\alpha> (1.1 \pi e)^{-1} \approx 0.1065$, $\alpha < (1+\frac{1}{e}-2a_0)/(1.1 \pi) \approx 0.1296$.
Thus, for $\alpha \in (0.1197,0.1296)$, Corollary~\ref{corollary2a} works and Proposition~\ref{proposition2a} fails.

In addition, let us note that Theorem~\ref{theorem1} and Corollary~\ref{corollary2a} can be applied to the equation 
\begin{equation}
\label{ex1eq2}
\left( x(t)- a(t)x(g(t)) \right)'=-\alpha (1+0.1\cos t)x(h(t)), \quad t-\pi \leq h(t) \leq t, ~t-\sigma \leq g(t) \leq t,
\end{equation}
leading to the same estimates as above, while Proposition~\ref{proposition2a} deals with constant delays only.
\end{example}

Most known stability results were obtained for pantograph equations involving a non-delay term.
For example, the  equation
\begin{equation}\label{pant}
\left( x(t)-a(t)x(\mu t) \right)'=-c(t)x(t)+b(t)x(\mu t), ~~\mu\in (0,1),
\end{equation}
where $0\leq \frac{c(\mu t)}{c(t)}a(t)\leq a_0<1$,  $c(t)\geq c_0>0$, is asymptotically stable  if 
$\frac{|b(t)|}{c(t)}\leq \alpha<1$ for some $\alpha>0$, as follows from \cite[P.286-287]{Bellen}, where the vector case 
was considered.
It means that the non-delay term dominates over the delay term.
This result partially generalizes  Proposition~\ref{proposition2b} for neutral equation (\ref{pant}). 

In this paper we considered equation (\ref{eq_pant}) without a non-delay term (in  (\ref{pant}) $c(t) \equiv 0$). 
Hence the results of the present paper and known stability tests for pantograph equations are independent.

\begin{example}
\label{ex2}
The pantograph-type neutral equation
\begin{equation}
\label{ex2eq1}
\left( x(t)- \frac{1}{3}x(0.25 t) \right)'=- \frac{1}{t} x(0.5t), \quad t \geq 1 
\end{equation}
is asymptotically stable, since
all the assumptions of Part b) of Corollary~\ref{cor_pant} 
hold.  In fact, 
$$\int_{0.5t}^t \frac{ds}{s} = \ln 2 \approx 0.693> \frac{1}{e}, \quad
\ln 2 < 1+\frac{1}{e}- \frac{2}{3} \approx 0.701.
$$
\end{example}

\begin{example}
\label{ex3}
For the equation with unbounded delays
\begin{equation}
\label{ex3eq1}
\left( x(t)- (0.1+0.1\sin t) x(t-\sqrt{t}) \right)'=- \frac{\alpha}{t \ln t} x\left( \sqrt{t} \right), \quad t \geq 4, ~\alpha>0,
\end{equation}
we have
$\displaystyle
\int_{\sqrt{t}}^t \frac{\alpha~ds}{s\ln s}= \alpha[\ln(\ln(t))-\ln(\ln(\sqrt{t}))] = \alpha \ln 2 \approx 0.693 \alpha.
$
Since $t>t-\sqrt{t} \geq \sqrt{t} $ for $t\geq 4$, also 
$\displaystyle
\alpha \int_{t-\sqrt{t}}^t \frac{ds}{s\ln s}
\leq \alpha \int_{\sqrt{t}}^t \frac{ds}{s\ln s} = \alpha \ln 2 <+\infty,
$ 
so (\ref{1a}) is satisfied. As $\| 0.1+0.1\sin t \|_{[4,+\infty)}=0.2$, a)  in Theorem~\ref{theorem2a}  holds for
$
\alpha \ln 2 \leq 1/e$, while b)  is fulfilled for
$$
\frac{1}{e} < \alpha \ln 2 <0.6+\frac{1}{e}.
$$
Overall, (\ref{ex3eq1}) is asymptotically stable for $\displaystyle \alpha < \frac{0.6 e+1}{e \ln 2} \approx 1.396$.
To the best of our knowledge, all known stability tests fail for this equation.
\end{example}

Let us discuss now both known results and new stability tests presented in the paper.
Proposition~\ref{proposition3} assumes existence of a non-delay term and thus cannot be applied to equation \eqref{1}.
Proposition~\ref{proposition1} contains easily verifiable conditions but implies several unnecessary restrictions, such as non-negativity and differentiability of $a$. 

Propositions~\ref{proposition2} and \ref{proposition2a} in the non-neutral case $a(t)\equiv 0$ give the best possible
asymptotic stability condition $\displaystyle \limsup_{t\rightarrow +\infty}\int_{t-\sigma}^t b(s)ds<\frac{3}{2}$,
but only for constant delays. We consider variable delays $t-h(t)$, $t-g(t)$ which, moreover, can be unbounded.

In all stability results of Propositions~\ref{proposition1}, \ref{proposition2}, \ref{proposition2a}, it was assumed that 
all the parameters of considered neutral equations are continuous functions,
and the proofs were based on this assumption. Thus all these results are not applicable to equations
with measurable parameters. 

Note that Theorem \ref{theorem2} for the case $a\equiv 0$ implies the best possible known stability
condition $\displaystyle  \tau \|b\|_{[t_0, +\infty)}<1+\frac{1}{e}$ for delay differential equations with one delay and measurable parameters.

Finally, let us suggest several directions in which future research is possible.
\begin{enumerate}
\item
An interesting question is whether in Theorem \ref{theorem1}  the condition 
$\displaystyle \|a\|_{[t_0, +\infty)}<\frac{1}{2}$ (as in Proposition~\ref{proposition1}) can be relaxed to a less restrictive inequality $\displaystyle \|a\|_{[t_0, +\infty)}<\lambda$, where $\displaystyle \lambda\in  \left( \frac{1}{2},1 \right)$. 
For a neutral equation in a different form than \eqref{1}, such a result was obtained in  \cite{BB2_2018}, under the assumption that $a(t) \geq 0$.

\item
Extend the stability result obtained in the paper to equations with several delays, integro-differential equations
and equations with distributed delays.

\item
There are many papers on asymptotic formulas for solutions of neutral equations, including pantograph
equations, see, for example, \cite{Bellen}. However, most results are concerned with autonomous equations or equations with constant delays.
It would be interesting to obtain similar estimates for non-autonomous equations using the Bohl-Perron theorem or another approach.

\item
Extend the results on the algebraic decay rate for pantograph equations to some other types of equations with unbounded delays, for example,
to $h(t)=\alpha \sqrt{t}$ or $h(t)= t- \alpha\sqrt{t}$, $t \geq 1$, $\alpha \in (0,1]$ and give an explicit estimate of this rate.

\end{enumerate}


\section*{Acknowledgment}

The second author was partially supported by the NSERC research grant RGPIN-2015-05976.
The authors are grateful to the anonymous referees whose valuable comments significantly contributed 
to the presentation of the results and the quality of the paper.

%
%

\end{document}